\pdfoutput=1
\documentclass[12pt]{amsart}
\usepackage[latin1]{inputenc}
\usepackage[english]{babel}
\usepackage[bookmarks=false]{hyperref}
\usepackage[usenames,dvipsnames]{xcolor}
\definecolor{dunkelblau}{rgb}{0,0,0.7}
\hypersetup{colorlinks=true,urlcolor=dunkelblau,linkcolor=dunkelblau,citecolor=dunkelblau}
\usepackage{amsfonts}
\usepackage{amssymb}
\usepackage{amsmath}
\usepackage{amsthm}
\usepackage{aliascnt}
\usepackage{dsfont}
\usepackage[left=3cm,right=3cm,top=2cm,bottom=2cm]{geometry}
\usepackage{mathtools}
\usepackage{mathrsfs}

\newtheoremstyle{standard}{9pt}{9pt}{\itshape}{}{\bfseries}{.}{.5em}{}
\theoremstyle{standard}
\newtheorem{lemma}{Lemma}[section]
\newtheorem{thm}[lemma]{Theorem}
\newtheorem{cor}[lemma]{Corollary} 
\newtheorem*{hal}{Theorem}

\newtheoremstyle{definition}{9pt}{9pt}{}{}{\bfseries}{.}{.5em}{}    
\theoremstyle{definition}
\newtheorem{defi}[lemma]{Definition}  
  
\newtheorem{rem}[lemma]{Remark}
\newtheorem{ex}[lemma]{Example}

\usepackage{tikz}
\usepackage{tikz-cd}
\tikzset{nodes={inner sep=2pt}}
\usetikzlibrary{decorations.markings}

\let\originalleft\left
\let\originalright\right
\renewcommand{\left}{\mathopen{}\mathclose\bgroup\originalleft}
\renewcommand{\right}{\aftergroup\egroup\originalright}

\newcommand\iso{\xrightarrow{\,\smash{\raisebox{-0.3ex}{\ensuremath{\scriptstyle\sim}}}\,}}
\newcommand\pfeil[1]{\xrightarrow{\;#1\;}}

\let\para\S

\newcommand{\IN}{\mathds{N}}
\newcommand{\IZ}{\mathds{Z}}
\newcommand{\IQ}{\mathds{Q}}

\newcommand{\IF}{\mathds{F}}

\newcommand{\IP}{\mathds{P}}

\newcommand{\C}{\mathcal{C}}

\newcommand{\X}{\mathcal{X}}
\newcommand{\V}{\mathcal{V}}
\renewcommand{\SS}{\mathcal{S}}

\renewcommand{\phi}{\varphi}
\renewcommand{\epsilon}{\varepsilon}

\newcommand{\bound}{\mathrm{bounded}}

\DeclareMathOperator{\p}{\mathfrak{p}}

\DeclareMathOperator{\Set}{\mathsf{Set}}
\DeclareMathOperator{\CAT}{\mathsf{CAT}}
\DeclareMathOperator{\Mod}{\mathsf{Mod}}
\DeclareMathOperator{\CoMod}{\mathsf{CoMod}}
\DeclareMathOperator{\gr}{\mathsf{gr}}
\DeclareMathOperator{\Qcoh}{\mathsf{Qcoh}}
\DeclareMathOperator{\SPEC}{\mathsf{Spec}}
\DeclareMathOperator{\FinSet}{\mathsf{FinSet}}
\DeclareMathOperator{\fp}{\mathsf{fp}}
\DeclareMathOperator{\fg}{\mathsf{fg}}
\DeclareMathOperator{\CAlg}{\mathsf{CAlg}}

\DeclareMathOperator{\Hom}{Hom}
\DeclareMathOperator{\HOM}{\underline{\Hom}}
\DeclareMathOperator{\End}{End}

\DeclareMathOperator{\Sym}{Sym}
\DeclareMathOperator{\colim}{colim}
\DeclareMathOperator{\id}{id}

\DeclareMathOperator{\op}{op}
\DeclareMathOperator{\Ann}{Ann}
\DeclareMathOperator{\proj}{\mathsf{proj}}
\DeclareMathOperator{\Spec}{Spec}

\begin{document}

\begin{abstract}\noindent
Let $k$ be a commutative $\IQ$-algebra. We study families of functors between categories of finitely generated $R$-modules which are defined for all commutative $k$-algebras $R$ simultaneously and are compatible with base changes. These operations turn out to be Schur functors associated to $k$-linear representations of symmetric groups. This result is closely related to Macdonald's classification of polynomial functors.
\end{abstract}

\bigskip

\title[Operations on categories of modules]{Operations on categories of modules\\are given by Schur functors}
\author{Martin Brandenburg}
\email{brandenburg@uni-muenster.de}

\maketitle


\section{Introduction}

\thispagestyle{empty}

Let $k$ be a commutative ring. We would like to understand functors between categories of finitely generated modules
\[F_R : \Mod_{\fg}(R) \to \Mod_{\fg}(R)\]
which are defined for all commutative $k$-algebras $R$ simultaneously and behave well with respect to base changes. This means that there are isomorphisms of $S$-modules
\[F_R(M) \otimes_R S \iso F_S(M \otimes_R S)\]
for $k$-homomorphisms $R \to S$ and finitely generated $R$-modules $M$, which are coherent in a suitable sense; see Definition \ref{defop} for details. We will call them \emph{operations} over $k$. Typical examples include the $n$-th tensor power $M \mapsto M^{\otimes n}$, the $n$-th symmetric power $M \mapsto \Sym^n(M)$ and the $n$-th exterior power $M \mapsto \Lambda^n(M)$. More generally, if $V_n$ is any finitely generated right $k[\Sigma_n]$-module, then
\[M \mapsto V_n \otimes_{k[\Sigma_n]} M^{\otimes n}\]
is an operation; here the symmetric group $\Sigma_n$ acts from the left on $M^{\otimes n}$ by permuting the tensor factors. Under finiteness conditions any direct sum of such operations is again an operation. We will call them \emph{Schur operations}, because if $k$ is a field and the $k$-algebra $R$ is fixed, or even $R=k$, these functors are called \emph{Schur functors} in representation theory \cite[Section 6.1]{FH91} and operad theory \cite[Section 5.1]{LV12}. In Joyal's theory of linear species \cite[Chapitre 4]{J86} they are called \emph{analytic functors}. Our main result states that, under suitable assumptions on $k$, every operation is a Schur operation.

\begin{hal}
Let $k$ be a commutative $\IQ$-algebra. Then every operation $\Mod_{\fg} \to \Mod_{\fg}$ over $k$ is isomorphic to a Schur operation
\[M \mapsto \bigoplus_{n \geq 0} V_n \otimes_{k[\Sigma_n]} M^{\otimes n}\]
for some sequence of finitely generated right $k[\Sigma_n]$-modules $V_n$. A similar classification holds for operations $\Mod_{\fp} \to \Mod_{\fp}$ between categories of finitely presented modules.
\end{hal}

See Theorem \ref{mainfull} for a description of those sequences $(V_n)$ which are allowed here. For example, finite sequences are allowed. Examples \ref{charp1}, \ref{charp2} and \ref{charp3} show why the  theorem fails in characteristic $p>0$.

A similar result was obtained by Macdonald \cite{M80} (see also \cite[Appendix A in Chapter I]{M79}) who considered \emph{polynomial functors} between categories of finitely generated modules over a \emph{fixed} commutative $\IQ$-algebra $R$, and gave a full classification for polynomial functors on finitely generated \emph{projective} modules. A similar classification therefore also holds for \emph{strict polynomial functors} \cite{BFS97, FS97, K13, T14} over a commutative $\IQ$-algebra $k$, which may be identified with operations $\Mod_{\fg,\proj} \to \Mod_{\fg,\proj}$ over $k$ between categories of finitely generated \emph{projective} modules. The classification in terms of Schur operations does not work for $k=\IZ$, but in this case one can at least calculate the $K_0$-ring of the category of strict polynomial functors \cite[Theorem 8.5]{HKT16}. 

We will adapt several techniques by Macdonald to our situation in order to reduce the classification from arbitrary operations first to \emph{homogeneous} and later to \emph{multilinear operations}; the latter are functors $\Mod_{\fg}(R)^n \to \Mod_{\fg}(R)$ which are $R$-linear in each variable, are defined for all commutative $k$-algebras $R$ simultaneously and behave well with respect to base changes. Their classification is quite simple.

\begin{hal}
Let $k$ be a commutative ring and $n \in \IN$. Then every multilinear operation $\Mod_{\fg}^n \to \Mod_{\fg}$ over $k$ is isomorphic to the operation
\[(M_1,\dotsc,M_n) \mapsto V \otimes_k (M_1 \otimes_R \dotsc \otimes_R M_n),\]
where $V$ is some finitely generated $k$-module.
\end{hal}
 
Although our proof uses finiteness assumptions in a crucial way, we conjecture that multilinear operations $\Mod^n \to \Mod$ have the same classification. For this it suffices to consider the case $n=1$, i.e.\ linear operations. The classification of all operations $\Mod \to \Mod$ would be a consequence.

The mentioned description of operations is actually the ``essential surjectivity'' part of the following equivalence of categories, which is our main result; here ``bounded'' refers to a certain finiteness condition introduced in Section \ref{sec:homo}.
 
\begin{hal}
Let $k$ be a commutative $\IQ$-algebra. Then the category of bounded operations $\Mod_{\fg} \to \Mod_{\fg}$ over $k$ is equivalent to the category of finite sequences $(V_n)_{n \in \IN}$ of finitely generated right $k[\Sigma_n]$-modules. A similar classification is true for bounded operations $\Mod_{\fp} \to \Mod_{\fp}$.
\end{hal}

See Theorem \ref{mainfull} for a description of the category of all operations $\Mod_{\fg} \to \Mod_{\fg}$, and see Remark \ref{mainn} for a description of operations $\Mod_{\fg}^n \to \Mod_{\fg}$ with $n$ arguments.
 
Our result may be interpreted as follows: Let $\IP = \coprod_{n \geq 0} \Sigma_n$ be the permutation groupoid. Then the category of bounded operations $\Mod_{\fp} \to \Mod_{\fp}$ over $k$ is equivalent to the category $\widehat{\IP}_k^{\fp}$ of finitely presented functors $\IP^{\op} \to \Mod_{\fp}(k)$. This is actually an equivalence of $k$-linear tensor (i.e.\ symmetric monoidal) categories, when $\widehat{\IP}_k^{\fp}$ is equipped with Day convolution. Since the permutation groupoid $\IP$ is the tensor category freely generated by a single object, it follows that $\widehat{\IP}_k^{\fp}$ is the finitely cocomplete $k$-linear tensor category freely generated by an object $X$ \cite[Remark 5.1.14]{B14}; for this reason we denote it by $\Mod_{\fp}(k)[X]$. It follows rather formally from the $2$-categorical Yoneda Lemma that the category of operations which are defined on \emph{all} finitely cocomplete $k$-linear tensor categories simultaneously and behave well with respect to base changes is equivalent to $\Mod_{\fp}(k)[X]$ (cf.\ \cite{BT16}). In this respect, our result says that, if $k$ is a $\IQ$-algebra, bounded operations on finitely presented modules over commutative $k$-algebras are ``universal enough'' to be defined on \emph{arbitrary} finitely cocomplete $k$-linear tensor categories. This is quite surprising. It is also quite interesting that the study of operations on modules is equivalent to the study of representations of symmetric groups.
 
The author's motivation to study operations on modules originates from the contravariant equivalence of $2$-categories between \emph{tensorial} stacks and \emph{stacky} tensor categories (cf.\ \cite[Section 3.4]{B14} and \cite[Section 3.5]{BC14}). It has the following finitely presented analogue (which has the advantage that some set-theoretic issues disappear): To every stack $\X$ over a commutative ring $k$, which we allow to be fibered in categories and also be possibly non-algebraic, we may associate the finitely cocomplete $k$-linear tensor category $\Qcoh_{\fp}(\X)$ of quasi-coherent modules of finite presentation. These are precisely the operations $\X \to \Mod_{\fp}$ over $k$. Conversely, to every finitely cocomplete $k$-linear tensor category $\C$ we may associate the stack $\SPEC_{\fp}(\C)$ over $k$ which maps a commutative $k$-algebra $R$ to the category $\SPEC_{\fp}(\C)(R) \coloneqq \Hom_{fc\otimes/k}(\C,\Mod_{\fp}(R))$ of finitely cocontinuous $k$-linear tensor functors from $\C$ into $\Mod_{\fp}(R)$. This defines a $2$-adjunction
\[\begin{tikzpicture}
\node (A) at (0,0) {$\{\text{stacks over }k\}$};
\node (B) at (7,0) {$\left\{\text{\begin{minipage}{4.1cm}finitely cocomplete $k$-linear tensor categories\end{minipage}}\right\}^{\op}\!\!\!.$};
\draw[->,line width=0.4pt] (1.5,0.2) to (4.45,0.2);
\draw[<-,line width=0.4pt] (1.5,-0.2) to (4.45,-0.2);
\node (C) at (3,0) {$\scriptstyle\perp$};
\node (D) at (3,0.47) {$\Qcoh_{\fp}$};
\node (D) at (3,-0.52) {$\SPEC_{\fp}$};
\end{tikzpicture}\]
We call a stack $\X$ over $k$ \emph{fp-tensorial} if the canonical morphism $\X \to \SPEC_{\fp}(\Qcoh_{\fp}(\X))$ is an equivalence, and a finitely cocomplete $k$-linear tensor category $\C$ is called \emph{fp-stacky} if the canonical morphism $\C \to \Qcoh_{\fp}(\SPEC_{\fp}(\C))$ is an equivalence. Thus, there is a contravariant equivalence of $2$-categories between fp-tensorial stacks over $k$ and fp-stacky finitely cocomplete $k$-linear tensor categories. Tensorial schemes and stacks have been investigated in \cite{B14,BC14,HR14}. It is natural to ask if the tensor category $\Mod_{\fp}(k)[X]$, the ``polynomial $2$-ring in one variable over $k$'' (cf.\ \cite{CJF13}), is fp-stacky. Its universal property implies $\SPEC_{\fp}(\Mod_{\fp}[X]) \simeq \Mod_{\fp}$. Therefore, $\Qcoh_{\fp}(\SPEC_{\fp}(\Mod_{\fp}(k)[X]))$ identifies with the category of operations $\Mod_{\fp} \to \Mod_{\fp}$ over $k$. The canonical tensor functor from $\Mod_{\fp}(k)[X]$ associates to every finite sequence of finitely presented $k[\Sigma_n]$-modules $V_n$ the corresponding Schur operation
\[M \mapsto \bigoplus_{n \in \IN} V_n \otimes_{k[\Sigma_n]} M^{\otimes n}.\]
Therefore, our result says that, if $k$ is a $\IQ$-algebra, this functor is fully faithful and that its essential image consists of bounded operations. In particular, $\Mod_{\fp}(k)[X]$ is not fp-stacky, since for example $\bigoplus_{n \in \IN} \Lambda^n$ is an operation on $\Mod_{\fp}$ which is not bounded. We may say that $\Mod_{\fp}(k)[X]$ is approximately fp-stacky.

These issues might disappear in the setting of arbitrary, not necessarily finitely presented, modules. However, we could not prove the classification of linear operations in this generality so far. We conjecture that $\Mod(k)[X]=\widehat{\IP}_k$ is stacky (from which it would follow that the stack $\Mod$ is tensorial), i.e., that it is equivalent to the category of operations $\Mod \to \Mod$ over $k$ via Schur operations. More generally, $\Mod(k)[X_1,\dotsc,X_n]$ should be stacky and hence be equivalent to the category of operations $\Mod^n \to \Mod$ over $k$.

Another interesting problem is to describe the category of operations $\CAlg \to \Mod$ from commutative algebras to modules. Every functor $V : \FinSet^{\op} \to \Mod(k)$, i.e.\ augmented symmetric simplicial $k$-module, induces such an operation via the coend
\[A \mapsto \int^{X \in \FinSet} V(X) \otimes_k A^{\otimes X}.\]
We may ask if this defines an equivalence $[\FinSet^{\op},\Mod(k)] \iso [\CAlg,\Mod]$. Since $[\FinSet^{\op},\Mod(k)]$ is the cocomplete $k$-linear tensor category freely generated by a commutative algebra object \cite{G01}, this question asks if $[\FinSet^{\op},\Mod(k)]$ is stacky.

\section*{Organization of the paper}
 
The paper is organized as follows: In Section \ref{sec:op} we give the basic definitions concerning operations. In Section \ref{sec:multi} we classify linear and multilinear operations. In Section \ref{sec:homo} we show that every operation decomposes uniquely into homogeneous operations. In Section \ref{sec:red} we discuss the linearization of a homogeneous operation and use it to prove our main theorems. In Section \ref{sec:con} we give a constructive proof of the classification of linear operations.
 
\section*{Acknowledgements}
I am very grateful to Alexandru Chirvasitu, who made major contributions to this paper in its early stages, particularly to Section \ref{sec:homo}. I also would like to thank Ingo Blechschmidt for sharing his beautiful constructive proofs appearing in Section \ref{sec:con}, Oskar Braun for his careful proofreading of the preprint, Bernhard Köck for drawing my attention to \cite{HKT16}, as well as Antoine Touz\'{e} and Ivo Dell'Ambrogio for helpful discussions on the subject at the University of Lille. Finally I would like to thank the anonymous referee for making several suggestions for improvement.
 

\section{Operations} \label{sec:op}

If $R$ is a ring, we will denote by $\Mod(R)$ the category of right $R$-modules. The full subcategories of finitely generated resp.\ finitely presented right $R$-modules are denoted by $\Mod_{\fg}(R)$ resp.\ $\Mod_{\fp}(R)$. As usually, when $R$ is commutative, we will not always distinguish between left and right modules.
  
\begin{defi} \label{defop}
Let $k$ be a commutative ring. An \emph{operation} $\Mod \to \Mod$ over $k$ is a family of functors
\[F_R : \Mod(R) \to \Mod(R),\]
defined for every commutative $k$-algebra $R$, equipped with isomorphisms of $S$-modules
\[\theta_{f}(M) : F_R(M) \otimes_R S \iso F_S(M \otimes_R S),\]
for each $k$-homomorphism $f : R \to S$, natural in $M \in \Mod(R)$, such that the following two coherence conditions are satisfied: the diagram
\[\begin{tikzcd}[row sep=25pt]
F_R(M) \otimes_R R \ar{rr}{\theta_{\id_R}(M)} && F_R(M \otimes_R R)  \\ & F_R(M) \ar{ur}[swap]{\cong} \ar{ul}{\cong} & 
\end{tikzcd}\]
commutes, and for all $k$-homomorphisms $f: R \to S$, $g: S \to T$ the diagram
\[\begin{tikzcd}[column sep=55pt,row sep=25pt]
(F_R(M) \otimes_R S) \otimes_S T \ar{r}{\theta_f(M) \otimes_S T} \ar{d}[swap]{\cong} & F_S(M \otimes_R S) \otimes_S T \ar{r}{\theta_g(M \otimes_R S)} & F_T((M \otimes_R S) \otimes_S T) \ar{d}{\cong} \\
F_R(M) \otimes_R T \ar{rr}{\theta_{g \circ f}(M)}  && F_T(M \otimes_R T)
\end{tikzcd}\]
commutes.

A \emph{morphism of operations} $(F,\theta) \to (F',\theta')$ is defined as a family of morphisms of functors $\alpha_R : F_R \to F'_R$, defined for every commutative $k$-algebra $R$, such that for all $k$-homomorphisms $f: R \to S$ and all $R$-modules $M$ the diagram
\[\begin{tikzcd}[column sep = 40pt,row sep=25pt]
F_R(M) \otimes_R S \ar{r}{\theta_f(M)} \ar{d}[swap]{\alpha_R(M) \otimes S} & F_S(M \otimes_R S) \ar{d}{\alpha_S(M \otimes_R S)} \\
F'_R(M) \otimes_R S \ar{r}{\theta'_f(M)} & F'_S(M \otimes_R S)
\end{tikzcd}\]
commutes.

Usually we will suppress the isomorphisms $\theta_f$ from the notation. Operations of the form $\Mod_{\fg} \to \Mod_{\fg}$ and $\Mod_{\fp} \to \Mod_{\fp}$ over $k$ are defined in a similar way. For fixed $n \in \IN$, operations with $n$ arguments $F : \Mod^n \to \Mod$ are defined in a similar way, including the variants with $\Mod_{\fg}$ and $\Mod_{\fp}$. Here, we require natural isomorphisms
\[F_R(M_1,\dotsc,M_n) \otimes_R S \iso F_S(M_1 \otimes_R S,\dotsc,M_n \otimes_R S).\]
\end{defi}

\begin{ex}
Let $n \in \IN$. The most basic example of an operation is the family of functors
\[M \mapsto M^{\otimes_R \, n} \coloneqq M \otimes_R \dots \otimes_R M,\]
equipped with the canonical isomorphisms
\[M^{\otimes_R \,  n} \otimes_R S \iso (M \otimes_R S)^{\otimes_S \,  n},\]
mapping $(m_1 \otimes \dots \otimes m_n) \otimes s$ to $(m_1 \otimes 1) \otimes \dots \otimes (m_n \otimes 1) \cdot s$. For $n=0$ this is the constant operation $M \mapsto R$, and for $n=1$ the identity operation $M \mapsto M$. The most basic example of an operation with two arguments is the tensor product $(M,N) \mapsto M \otimes_R N$ equipped with the canonical isomorphisms $(M \otimes_R N) \otimes_R S \iso (M \otimes_R S) \otimes_S (N \otimes_R S)$.
\end{ex}

\begin{ex}
If $I$ is any set, then the direct sum $M \mapsto M^{\oplus I}$ defines an operation over $k$, but the direct product $M \mapsto M^I$ does not unless $k$ is the zero ring or $I$ is finite.
\end{ex}

The next three examples illustrate why our classification result requires that $k$ is a commutative $\IQ$-algebra.

\begin{ex} \label{charp1}
If $k$ is any commutative ring, then the second exterior power
\[M \mapsto \Lambda^2(M) = M^{\otimes 2} / \langle m \otimes m : m \in M \rangle\]
defines an operation over $k$. If $2$ is invertible in $k$, then we have
\[\Lambda^2(M) = M^{\otimes 2} / \langle m \otimes n + n \otimes m : m,n \in M \rangle \cong V \otimes_{k[\Sigma_2]} M^{\otimes 2},\]
where $V$ denotes the alternating representation of $\Sigma_2$ of rank $1$. However, when $k$ is a field of characteristic $2$, we cannot find a representation $V$ of $\Sigma_2$ such that there is an isomorphism of operations $\Lambda^2(M) \cong V \otimes_{k[\Sigma_2]} M^{\otimes 2}$. This also reflects the observation that we cannot define $\Lambda^2$ in arbitrary finitely cocomplete symmetric monoidal $k$-linear categories unless $2$ is invertible in $k$ \cite[Remark 4.4.7]{B14}.
\end{ex}

\begin{ex} \label{charp2}
If $k$ is any commutative ring, then the divided power algebra \cite[Chapitre III]{R63} $M \mapsto \Gamma(M)$ as well as its homogeneous parts $M \mapsto \Gamma_d(M)$ provide  operations over $k$; the base change isomorphisms are constructed in \cite[Th\'{e}or\`{e}me III.3]{R63}. If $k$ is a commutative $\IQ$-algebra, then there is an isomorphism of operations $\Gamma_d \cong \Sym^d$, but this is false if $k$ is a field of characteristic $p>0$.
\end{ex}

\begin{ex} \label{charp3}
Assume that $k$ is a commutative $\IF_p$-algebra. Then we can define an operation over $k$ by mapping an $R$-module $M$ to $M \otimes_{R,\phi_R} R$, where $\phi_R : R \to R$ is the Frobenius homomorphism. For a fixed $k$-algebra $R$ this is usually called the Frobenius twist \cite[Section 1]{FS97}. The base change isomorphisms for $f : R \to S$ are given by
\[(M \otimes_{R,\phi_R} R) \otimes_R S \iso M \otimes_{R,f \circ \phi_R} S = M \otimes_{R,\phi_S \circ f} S \iso (M \otimes_R S) \otimes_{S,\phi_S} S.\]
\end{ex}
 
\begin{rem}\label{onR}
If $F : \Mod \to \Mod$ is an operation over $k$ and $R$ is a commutative $k$-algebra, then we have
\[F_R(R) \cong F_R(k \otimes_k R) \cong F_k(k) \otimes_k R.\]
If $f \in R$ and $M$ is some $R$-module, then we have
\[F_R(M)[f^{-1}] \cong F_{R[f^{-1}]}\bigl(M[f^{-1}]\bigr).\]
If $G$ is a commutative monoid, then we use the notation $M[G] \coloneqq M \otimes_R R[G]$, and we have
\[F_R(M)[G] \cong F_{R[G]}\bigl(M[G]\bigr).\]
For the special case $G=\IN^d$ this becomes
\[F_R(M)[T_1,\dotsc,T_d] \cong F_{R[T_1,\dotsc,T_d]}\bigl(M[T_1,\dotsc,T_d]\bigr).\]
\end{rem}

\begin{rem}
We may rephrase the definition of an operation using $2$-categorical language \cite{B67} as follows. There is a pseudofunctor
\[\Mod : \CAlg(k) \to \CAT\]
which associates to every commutative $k$-algebra $R$ its category of $R$-modules $\Mod(R)$ and to every $k$-homomorphism $R \to S$ the base change functor $\square \otimes_R S$ equipped with the usual coherence isomorphisms; $\square$ denotes a placeholder. Here, $\CAT$ denotes the ``$2$-category'' of categories; one may use universes in order to make this precise. Then an operation $\Mod \to \Mod$ is just a pseudonatural transformation $\Mod \to \Mod$, and a morphism of operations is just a modification between them. A similar statement holds for operations $\Mod^n \to \Mod$ and the variants $\Mod_{\fg}$ and $\Mod_{\fp}$.
\end{rem}

\begin{rem}
In algebraic geometry, pseudonatural transformations $\X \to \Mod$ are sometimes called quasi-coherent modules on $\X$, especially when $\X$ is a stack (say, in the \'{e}tale topology on the category of affine $k$-schemes). If the target is $\Mod_{\fg}$ or $\Mod_{\fp}$, these quasi-coherent modules are called of finite type resp.\ of finite presentation. The pseudofunctors $\Mod$, $\Mod_{\fg}$ and $\Mod_{\fp}$ are stacks by descent theory \cite{V05}, and operations on them are just quasi-coherent modules (of finite type, resp.\ of finite presentation) on themselves.
\end{rem}

\begin{rem}
Up to size issues, there is a ``category'' of operations over $k$, which we will denote by $[\Mod,\Mod]$. As with every category of quasi-coherent modules, this is actually a small-cocomplete symmetric monoidal $k$-linear category (which includes, by definition, that $\otimes$ is cocontinuous and $k$-linear in each variable). Colimits and tensor products are computed pointwise:
\begin{align*}
(\colim_i F_i)_R(M) &= \colim_i \bigl((F_i)_R(M)\bigr),\\
1_R(M) &= R,\\
(F \otimes G)_R(M) &= F_R(M) \otimes F_R(M).
\end{align*}
Operations have an additional structure, given by composition:
\[(F \circ G)_R(M) = F_R(G_R(M)).\]
This defines another monoidal structure on $[\Mod,\Mod]$, which is however not symmetric and not cocomplete. We remark that the corresponding monoidal structure on the category of linear species $[\IP^{\op},\Mod(k)]$ is the substitution product \cite[Chapitre 4]{J86}. The ``category'' $[\Mod^n,\Mod]$ of operations $\Mod^n \to \Mod$ is also a small-cocomplete symmetric monoidal $k$-linear category, but it has no composition when $n>1$. The ``categories'' $[\Mod_{\fg}^n,\Mod_{\fg}]$ and $[\Mod_{\fp}^n,\Mod_{\fp}]$ are defined similarly.
\end{rem}

\begin{rem}
If is not clear a priori if $[\Mod,\Mod]$ is small-complete, and if so how the limits are computed. It is even less clear if $[\Mod,\Mod]$ is an abelian category.
\end{rem}

\begin{rem}
If $\Mod_{\fg,\proj} : \CAlg(k) \to \CAT$ denotes the pseudofunctor of finitely generated projective modules, there is a restriction functor
\[[\Mod_{\fg},\Mod_{\fg}] \to [\Mod_{\fg,\proj},\Mod_{\fg}].\]
It has no a priori reason to be fully faithful, let alone to be an equivalence. The category of strict polynomial functors by Friedlander and Suslin may be identified with $[\Mod_{\fg,\proj},\Mod_{\fg,\proj}]$ by using \cite[Proposition 2.5]{BFS97}. Some authors \cite{K13,T14} allow the codomain to be $\Mod$, but the domain in the theory of strict polynomial functors has always been $\Mod_{\fg,\proj}$. One can use \cite[Theorem 2.14]{B00} to define an extension functor
\[[\Mod_{\proj},\Mod] \to [\Mod,\Mod],\]
which again has no a priori reason to be an equivalence. Therefore the category of operations is a priori just a variant of the category of strict polynomial functors, and a description of one category does not directly imply a description of the other one.
\end{rem}

There is another difference between operations and strict polynomial functors.

\begin{rem} \label{enrich}
The components $F_R : \Mod(R) \to \Mod(R)$ of an operation $F$ admit an enrichment in the cartesian monoidal category of functors $[\CAlg(R),\Set]$ as follows: We may view $\Mod(R)$ as a category enriched in $[\CAlg(R),\Set]$ by defining, for every pair of $R$-modules $(M,N)$, the functor $\HOM_R(M,N) : \CAlg(R) \to \Set$ on objects $S$ by
\[\HOM_R(M,N)(S) \coloneqq \Hom_S(M \otimes_R S,N \otimes_R S).\]
The map of sets $F_R : \Hom_R(M,N) \to \Hom_R(F_R(M),F_R(N))$ extends to a natural transformation $\HOM_R(M,N) \to \HOM_R(F_R(M),F_R(N))$ via mapping an $S$-linear map $f : M \otimes_R S \to N \otimes_R S$ to the $S$-linear map
\[F_R(M) \otimes_R S \iso F_S(M \otimes_R S) \xrightarrow{F_S(f)} F_S(N \otimes_R S) \iso F_R(N) \otimes_R S.\]
This extra structure of $F_R$ will be very useful in Section \ref{sec:homo}. It is similar, but not identical to the enrichment in the definition of a strict polynomial functor \cite{T14}: Here one requires maps of sets $\Hom_R(M,N) \otimes_R S \to \Hom_R(F_R(M),F_R(N)) \otimes_R S$ which are natural in $S$, i.e.\ that $\Hom_R(M,N) \to \Hom_R(F_R(M),F_R(N))$ is a polynomial rule (``lois polynomes'') in the sense of Roby \cite{R63}. The canonical map
\[\Hom_R(M,N) \otimes_R S \to \Hom_S(M \otimes_R S,N \otimes_R S)\]
is bijective if $M$ is finitely generated projective, but in general it is neither injective nor surjective.
\end{rem}
 
 
\section{Linear and multilinear operations} \label{sec:multi}
  
Linear and multilinear operations provide a basic class of operations which we are going to classify first.
  
\begin{defi}\label{linear-def}
Let $k$ be a commutative ring. An operation $F : \Mod_{\fg} \to \Mod_{\fg}$ over $k$ is called \emph{linear} if for every commutative $k$-algebra $R$ the functor
\[F_R : \Mod_{\fg}(R) \to \Mod_{\fg}(R)\]
is $R$-linear. This defines a full subcategory
\[[\Mod_{\fg},\Mod_{\fg}]_1  \subseteq [\Mod_{\fg},\Mod_{\fg}].\]
It contains the identity operation and is closed under composition, therefore may be regarded as a monoidal $k$-linear category. The category $[\Mod_{\fp},\Mod_{\fp}]_1$ is defined in a similar way.
\end{defi}
 
\begin{ex}
Let $V$ be some finitely generated $k$-module. Then $M \mapsto V \otimes_k M$ becomes a linear operation using the natural isomorphisms
\[ (V \otimes_k M) \otimes_R S \iso V \otimes_k (M \otimes_R S)\]
for $k$-homomorphisms $R \to S$. In fact, this construction induces a $k$-linear functor
\[\Mod_{\fg}(k) \to [\Mod_{\fg},\Mod_{\fg}]_1, \quad V \mapsto V \otimes_k \square.\]
We may equip this functor with the structure of a monoidal functor by using the natural isomorphisms $k \otimes_k \square \iso \square$ and $(V \otimes_k W) \otimes_k \square \iso V \otimes_k (W \otimes_k \square)$.
\end{ex}

\begin{thm}\label{lin}
The monoidal functor constructed above yields an equivalence of mon\-oidal $k$-linear categories
\[\Mod_{\fg}(k) \simeq [\Mod_{\fg},\Mod_{\fg}]_1.\]
The same construction yields
\[\Mod_{\fp}(k) \simeq [\Mod_{\fp},\Mod_{\fp}]_1.\]
\end{thm}

\begin{proof}
We construct an explicit pseudo-inverse functor. Let $F : \Mod_{\fg} \to \Mod_{\fg}$ be a linear operation over $k$. We associate to it the finitely generated $k$-module $V \coloneqq F_k(k)$. If $M$ is some finitely generated $R$-module, then there is a natural $R$-linear map
\begin{align*}
M \iso \Hom_R(R,M) & \xrightarrow{\phantom{\,\sim\,}} \Hom_R(F_R(R),F_R(M)) \\
 & \iso   \Hom_R(V \otimes_k R,F_R(M))  \iso \Hom_{k}(V,F_R(M)|_k),
\end{align*}
which corresponds to a natural $R$-linear map
\[\alpha_R(M) : V \otimes_k M \to F_R(M).\]
This is, in fact, a morphism of linear operations $\alpha : V \otimes_k \square \to F$, as can be checked from the coherence condition in Definition \ref{defop} applied to $k \to R \to S$. We have to show that it is an isomorphism. By Remark \ref{onR}, it is an isomorphism when $M=R$. Since linear operations are additive, it is also an isomorphism when $M$ is a finitely generated free $R$-module.

We first show that $\alpha_R(M)$ is an epimorphism for every $R$-module $M$. By taking the cokernel of $\alpha$, which is a linear operation again, it suffices to prove the following: If $G$ is a linear operation which vanishes on finitely generated free modules, then $G=0$. Take any finitely generated $R$-module $M$. If $\p$ is any prime ideal of $R$, we have
\[G_{Q(R/\p)}(M \otimes_R Q(R/\p))=0\]
since $Q(R/\p)$ is a field and therefore $M \otimes_R Q(R/\p)$ is free. It follows that
\[0 = G_R(M) \otimes_R Q(R/\p) \cong G_R(M)_{\p} / \p G_R(M)_{\p}.\]
Nakayama's Lemma implies $G_R(M)_{\p} = 0$. Since $\p$ was arbitrary, this shows $G_R(M)=0$.
 
It remains to prove that $\alpha_R(M) : V \otimes_k M \to F_R(M)$ is injective. If this happens to be the case, let us call $M$ \emph{good}. Since linear functors are additive, direct summands of good modules are good. Since $M$ is a direct summand of $M \oplus R$, which has the structure of a commutative $R$-algebra in which $M$ squares to zero, we may assume that $M$ is the underlying $R$-module of a commutative $R$-algebra $S$. More generally, we assume that $M=N|_R$ is the underlying $R$-module of some good $S$-module $N$, where $S$ is a commutative $R$-algebra. Consider the following commutative diagram:
\[\begin{tikzcd}[row sep=30pt, column sep=34pt]
V \otimes_k M \ar{d}[swap]{\alpha} \ar{r}{\text{unit}} & ((V \otimes_k M) \otimes_R S)|_R \ar{r}{\sim} \ar{d}[swap]{\alpha} & (V \otimes_k (M \otimes_R S))|_R \ar{r}{\text{counit}} \ar{d}[swap]{\alpha} & (V \otimes_k N)|_R \ar{d}[swap]{\alpha} \\
F_R(M) \ar{r}{\text{unit}} & (F_R(M) \otimes_R S)|_R \ar{r}{\sim} & F_S(M \otimes_R S)|_R \ar{r}{\text{counit}} & F_S(N)|_R
\end{tikzcd}\]
Here unit and counit refer to the adjunction between extending and restricting scalars. The composition $V \otimes_k M \to (V \otimes_k N)|_R$ is the identity, and $\alpha : (V \otimes_k N)|_R \to F_S(N)|_R$ is an isomorphism since $N$ is a good $S$-module. Hence, the diagram implies that $\alpha : V \otimes_k M \to F_R(M)$ is a split monomorphism. Therefore, $M$ is a good $R$-module.
\end{proof}

\begin{rem}
The usage of Nakayama's Lemma in the proof above is the only reason why we have restricted ourselves to finitely generated modules.
\end{rem}

\begin{rem}
One might wonder if there is a constructive proof of Theorem \ref{lin}, which does neither use the existence of prime ideals nor the law of the excluded middle. This is indeed possible and will be shown in Section \ref{sec:con}.
\end{rem}

\begin{defi}\label{multilindef}
Let $k$ be a commutative ring and $n \in \IN$. An operation with $n$ arguments $F : \Mod_{\fg}^n \to \Mod_{\fg}$ will be called \emph{multilinear} if it is linear in each variable. This means that for every index $1 \leq i \leq n$ and every family of finitely generated $R$-modules $M_1,\dotsc,M_{i-1},M_{i+1},\dotsc,M_n$ the functor
\[F_R(M_1,\dotsc,M_{i-1},\square,M_{i+1},\dotsc,M_n) : \Mod_{\fg}(R) \to \Mod_{\fg}(R)\]
is $R$-linear. This defines a full subcategory
\[ [\Mod_{\fg}^n,\Mod_{\fg}]_{1,\dotsc,1} \subseteq [\Mod_{\fg}^n,\Mod_{\fg}].\]
For $n=1$ we recover $[\Mod_{\fg},\Mod_{\fg}]_1$.
\end{defi}

\begin{ex}
Let $V$ be some finitely generated $k$-module. Then
\[\Mod_{\fg}(R)^n \to \Mod_{\fg}(R), \quad (M_1,\dotsc,M_n) \mapsto V \otimes_k (M_1 \otimes_R \dots \otimes_R M_n)\]
becomes a multilinear operation using the natural isomorphisms
\[ (V \otimes_k (M_1 \otimes_R \dots \otimes_R M_n)) \otimes_R S \iso V \otimes_k ((M_1 \otimes_R S) \otimes_S \dots \otimes_S (M_1 \otimes_R S)).\]
In fact, this induces a functor $\Mod_{\fg}(k) \to [\Mod_{\fg}^n,\Mod_{\fg}]_{1,\dotsc,1}$.
\end{ex}

\begin{thm} \label{multilin}
The functor constructed above yields an equivalence of categories
\[\Mod_{\fg}(k) \simeq [\Mod_{\fg}^n,\Mod_{\fg}]_{1,\dotsc,1}.\]
The same construction yields
\[\Mod_{\fp}(k) \simeq [\Mod_{\fp}^n,\Mod_{\fp}]_{1,\dotsc,1}.\]
\end{thm}
 
\begin{proof}
As in Theorem \ref{lin}, we construct a pseudo-inverse functor by mapping a multilinear operation $F : \Mod_{\fg}^n \to \Mod_{\fg}$ to the finitely generated $k$-module $V \coloneqq F_k(k,\dotsc,k)$, and the natural $R$-linear maps
\begin{align*}
M_1 \otimes_R \dots \otimes_R M_n&  \iso \Hom_R(R,M_1) \otimes_R \dots \otimes_R \Hom_R(R,M_n) \\
 & \pfeil{~} \Hom_R\bigl(F_R(R,\dotsc,R),F_R(M_1,\dotsc,M_n)\bigr) \\
 & \iso \Hom_k(V,F_R(M_1,\dotsc,M_n)|_k
\end{align*}
induce a morphism of multilinear operations
\[V \otimes_k (M_1 \otimes_R \dots \otimes_R M_n) \to F_R(M_1,\dotsc,M_n).\]
It suffices to prove that it is an isomorphism. We will argue by induction on $n$, the case $n=0$ being trivial. Now let us assume $n \geq 1$ and that the theorem is true for $n-1$. Let us fix some commutative $k$-algebra $R$ and some $R$-module $M_1$. Then we may define a multilinear operation $\Mod_{\fg}^{n-1} \to \Mod_{\fg}$ \emph{over $R$} by
\[(M_2,\dotsc,M_n) \mapsto F_S(M_1 \otimes_R S,M_2,\dotsc,M_n)\]
for $S$-modules $M_2,\dotsc,M_n$, where $S$ is a commutative $R$-algebra. By induction hypothesis, the canonical homomorphism
\[F_R(M_1,R,\dotsc,R) \otimes_R (M_2 \otimes_S \dots \otimes_S M_n) \to F_S(M_1 \otimes_R S,M_2,\dotsc,M_n)\]
is an isomorphism. We will only need this for $S=R$. Varying $R$ and $M_1$, we observe that $M_1 \mapsto F_R(M_1,R,\dotsc,R)$ defines a linear operation over $k$. Thus, by Theorem \ref{lin}, the canonical homomorphism
\[V \otimes_k M_1 \to F_R(M_1,R,\dotsc,R)\]
is an isomorphism. We finish the proof by composing this isomorphism with the previous one.
\end{proof}

\begin{rem}
We have shown rather indirectly that multilinear operations are right exact in each variable, and we have found a characterization of tensor products which does not involve right exactness in any way, but rather base change. Notice that the Eilenberg-Watts Theorem \cite{E60} would immediately imply the classification of linear operations $\Mod \to \Mod$ (resp.\ $\Mod_{\fp} \to \Mod_{\fp}$) if we already knew that they consisted of cocontinuous (resp.\ right exact) functors.
\end{rem}


\section{Homogeneous operations} \label{sec:homo}

In this section we will show that every operation $\Mod \to \Mod$ over a commutative ring $k$ decomposes uniquely into homogeneous operations. This is analogous to the homogeneous decomposition of strict polynomial functors \cite[\para 2]{FS97}.

\begin{lemma} \label{comod}
Let $R$ be a commutative $k$-algebra and let $A$ be a commutative $R$-bialgebra. Consider the category of comodules $\CoMod(A)$ with its forgetful functor to $\Mod(R)$. If $F : \Mod \to \Mod$ is an operation over $k$, then the functor $F_R : \Mod(R) \to \Mod(R)$ lifts to a functor $\CoMod(A) \to \CoMod(A)$.
\[\begin{tikzcd} \CoMod(A)  \ar{d} \ar[dashed]{r} & \CoMod(A) \ar{d} \\ \Mod(R) \ar{r}{F_R} & \Mod(R) \end{tikzcd}\]
\end{lemma}

\begin{proof}
Let $M$ be an $R$-module equipped with an $A$-coaction $h : M \to M \otimes_R A$. This coaction corresponds 1:1 to a family of monoid homomorphisms
\[\alpha_B : \Hom_{\CAlg(R)}(A,B) \to \End_{\Mod(B)}(M \otimes_R B)\]
for commutative $R$-algebras $B$ which are natural in $B$ \cite[II, \S 2, n$^\mathrm{o}\,$2]{DG70}. The monoid structure on $\Hom_{\CAlg(R)}(A,B)$ is induced by the commutative bialgebra structure of $A$. Specifically, $\alpha_B(f)$ is defined from $h$ by
\[M \otimes_R B \xrightarrow{\,h \otimes B\,} M \otimes_R A \otimes_R B \xrightarrow{M \otimes f  \otimes B} M \otimes_R B \otimes_R B \xrightarrow{M \otimes \mu_B} M \otimes_R B.\]
Conversely, a family $(\alpha_B)_{B \in \CAlg(R)}$ is mapped to the coaction
\[M  \xrightarrow{\,\eta_A\,} M \otimes_R A \xrightarrow{\alpha_A(\id_A)} M \otimes_R A.\]
Using this description of comodule structures, we obtain an $A$-coaction on $F_R(M)$ using the natural monoid homomorphisms (cf.\ Remark \ref{enrich})
\[ \End_{\Mod(B)}(M \otimes_R B) \to \End_{\Mod(B)}(F_B(M \otimes_R B)) \iso \End_{\Mod(B)}(F_R(M) \otimes_R B).\]
Specifically, the $R$-linear coaction $M \to M \otimes_R A$ corresponds to some $A$-linear map $M \otimes_R A \to M \otimes_R A$, which induces an $A$-linear map $F_R(M) \otimes_R A \to F_R(M) \otimes_R A$ (using the identification $F_R(M) \otimes_R A \cong  F_A(M \otimes_R A)$), which corresponds to some $R$-linear map $F_R(M) \to F_R(M) \otimes_R A$, the $A$-coaction on $F_R(M)$. If another $A$-comodule $h' : M' \to M' \otimes_R A$ is given, one easily checks that for every homomorphism $M \to M'$ of $A$-comodules the induced homomorphism $F_R(M) \to F_R(M')$ is also a homomorphism of $A$-comodules.
\end{proof}

\begin{defi}
If $G$ is a commutative monoid, let us denote by $\gr_G\!\Mod(R)$ the category of $G$-graded $R$-modules. If $M = \bigoplus_{g \in G} M_g$ is a graded $R$-module and $R \to S$ is a $k$-homomorphism, we endow $M \otimes_R S$ with the grading $M \otimes_R S = \bigoplus_{g \in G} M_g \otimes_R S$. This defines a pseudofunctor $\gr_G\!\Mod : \CAlg(k) \to \CAT$ together with a forgetful operation $\gr_G\!\Mod \to \Mod$.
\end{defi}
 
\begin{cor} \label{grad}
Let $G$ be a commutative monoid. Every operation $\Mod \to \Mod$ lifts, along the forgetful operation $\gr_G\!\Mod \to \Mod$, to an operation $\gr_G\!\Mod \to \gr_G\!\Mod$.
\[\begin{tikzcd} \gr_G\!\Mod  \ar{d} \ar[dashed]{r} & \gr_G\!\Mod \ar{d} \\ \Mod \ar{r}{F} & \Mod \end{tikzcd}\]
\end{cor}

\begin{proof}
This follows from Lemma \ref{comod} because $\gr_G\!\Mod(R)$ is isomorphic to the category of $R[G]$-comodules \cite[II, \S 2, 2.5]{DG70} when the monoid algebra $R[G]$ is equipped with the counit $g \mapsto 1$ and the comultiplication $g \mapsto g \otimes g$. Specifically, if $M=\bigoplus_{g \in G} M_g$ is $G$-graded, the corresponding coaction is the $R$-linear map
\[\textstyle M \to M[G], \quad \sum_{g \in G} m_g \mapsto \sum_{g \in G} m_g \cdot g.\]
By adjunction this corresponds to an $R[G]$-linear map $M[G] \to M[G]$, which induces an $R[G]$-linear map
\[F_R(M)[G] \to F_R(M)[G].\]
The homogeneous component $F_R(M)_g \subseteq F_R(M)$ of degree $g\in G$ is the submodule consisting of those elements $u \in F_R(M)$ such that $u \cdot 1 \in F_R(M)[G]$ gets mapped to $u \cdot g \in F_R(M)[G]$. If $R \to S$ is a $k$-homomorphism, one has to check that the isomorphism $\theta : F_R(M) \otimes_R S \to F_S(M \otimes_R S)$ preserves the gradings. This follows from the coherence conditions in Definition \ref{defop} applied to $R \to R[G] \to S[G]$ and $R \to S \to S[G]$.
\end{proof}
 
\begin{defi}
Let $F : \Mod \to \Mod$ be an operation. By Corollary \ref{grad}, $F$ lifts to an operation $\gr_{\IN}\!\Mod \to \gr_{\IN}\!\Mod$. There is a canonical operation $\Mod \to \gr_{\IN}\!\Mod$ which equips every module with the trivial grading concentrated in degree $1$. The composition $\Mod \to \gr_{\IN}\!\Mod \to \gr_{\IN}\!\Mod$ corresponds to a family of operations $F_n : \Mod \to \Mod$ with an isomorphism of operations
\[\bigoplus_{n \in \IN} F_n \iso F.\]
We call $F_n$ the \emph{homogeneous component of degree $n$} of $F$. Specifically, if $M$ is some $R$-module, then $(F_n)_R(M)$ consists of those elements $u \in F_R(M)$ such that $F_{R[T]}\bigl(T \cdot \id_{M[T]}\bigr)$ maps $u \cdot 1$ to $u \cdot T^n \in F_R(M)[T]$ (using the identification $F_{R[T]}(M[T]) \cong F_R(M)[T]$).
\end{defi}
 
\begin{defi} \label{homo-def}
Let $n \in \IN$. An operation $F : \Mod \to \Mod$ is called \emph{homogeneous of degree $n$} if for every $R$-module $M$ we have
\[F_{R[T]}\bigl(T \cdot \id_{M[T]}\bigr) = T^n \cdot \id_{F_{R[T]}(M[T])}.\]
Let $[\Mod,\Mod]_n \subseteq [\Mod,\Mod]$ denote the full subcategory of operations which are homogeneous of degree $n$.
\end{defi}
 
\begin{cor} \label{decomp1}
There is an equivalence of categories
\[\prod_{n \in \IN} [\Mod,\Mod]_n \to [\Mod,\Mod], \quad (F_n)_{n \in \IN} \mapsto \bigoplus_{n \in \IN} F_n.\]
\end{cor}

\begin{proof}
We define a pseudo-inverse functor by $F \mapsto (F_n)_{n \in \IN}$, where $F_n$ is the homogeneous component of degree $n$ of $F$. Using flatness of $R \to R[T]$, it follows easily that $F_n$ is, in fact, homogeneous of degree $n$. If $F,F'$ are two operations, using the compatibility with the base change $R \to R[T]$, one checks that every morphism $F \to F'$ restricts to a morphism $F_n \to F'_n$, where $n \in \IN$ is arbitrary. Finally, notice that the homogeneous component of degree $n$ of $\bigoplus_{n \in \IN} F_n$ is precisely $F_n$.
\end{proof}

Using similar definitions for finitely generated modules, we obtain:

\begin{cor} \label{decomp2}
There is a fully faithful functor
\[[\Mod_{\fg},\Mod_{\fg}] \to \prod_{n \in \IN} [\Mod_{\fg},\Mod_{\fg}]_n, \quad F \mapsto (F_n)_{n \in \IN}.\]
Its essentially image consists of those families $(F_n)_{n \in \IN}$ of operations, homogeneous of degree $n$, such that for every finitely generated $R$-module $M$ almost all images $(F_n)_R(M)$ vanish; in other words, $\bigoplus_{n \in \IN} (F_n)_R(M)$ is supposed to be finitely generated.
\end{cor}

\begin{ex}
The $n$-th exterior power $\Lambda^n$ is an operation which is homogeneous of degree $n$. The direct sum $\bigoplus_{n \in \IN} \Lambda^n$ is an operation both on $\Mod$ and on $\Mod_{\fg}$. This is because if some $R$-module $M$ is generated by $n$ elements, then $\Lambda^m_R(M)=0$ for all $m>n$. This shows that there are operations on $\Mod_{\fg}$ with infinitely many non-trivial homogeneous components.
\end{ex}

\begin{defi} \label{bounded}
Let us call an operation $F : \Mod_{\fg} \to \Mod_{\fg}$ \emph{bounded} if there is some $n \in \IN$ such that $F_m=0$ for all $m>n$. We get a full subcategory $[\Mod_{\fg},\Mod_{\fg}]_{\bound}$ of $[\Mod_{\fg},\Mod_{\fg}]$.
\end{defi}

\begin{rem} \label{bounded2}
By Corollary \ref{decomp2}, we have
\[[\Mod_{\fg},\Mod_{\fg}]_{\bound} \simeq \bigoplus_{n \in \IN}[\Mod_{\fg},\Mod_{\fg}]_n.\]
Here, we use the notation $\bigoplus_{i \in I} \C_i$ for the full subcategory of $\prod_{i \in I} \C_i$ consisting of those objects $(X_i)_{i \in I}$ such that almost all $X_i$ are zero.
\end{rem}

Corollaries \ref{decomp1} and \ref{decomp2} allow us to restrict our attention to the categories of homogeneous operations $[\Mod,\Mod]_n$ resp.\ $[\Mod_{\fg},\Mod_{\fg}]_n$ for some fixed value of $n \in \IN$.

In the remainder of this section, each time $\Mod$ may be replaced by $\Mod_{\fg}$.

\begin{lemma}\label{homo-krit}
Let $F : \Mod \to \Mod$ be an operation. Then $F$ is homogeneous of degree~$n$ if and only if for every commutative $k$-algebra $R$, every $R$-module $M$ and every element $r \in R$ we have $F_R\bigl(r \cdot \id_M\bigr) = r^n \cdot \id_{F_R(M)}$.
\end{lemma}

\begin{proof}
The direction $\Longleftarrow$ follows by applying the assumption to the $R[T]$-module $M[T]$ and the element $T \in R[T]$. The direction $\Longrightarrow$ follows by applying the base change $R[T] \to R$, $T \mapsto r$ to the assumption $F_{R[T]}\bigl(T \cdot \id_{M[T]}\bigr) =  T^n \cdot \id_{F_R(M)[T]}$.
\end{proof}
 
Our next result shows that homogeneous operations consist of polynomial functors in the sense of \cite[Sections 1 and 2]{M80}. It also shows the connection to Roby's polynomial rules (cf.\ \cite[Th\'{e}or\`{e}me I.1]{R63} and Remark \ref{enrich}).
 
\begin{lemma} \label{poly}
Let $F:\Mod \to \Mod$ be an operation which is homogeneous of degree $n$. If $M,N$ are $R$-modules, then the map
\[F_R : \Hom_R(M,N) \to \Hom_R\bigl(F_R(M),F_R(N)\bigr)\]
has the following property: Let $d \in \IN$ and $f_1,\dotsc,f_d \in \Hom_R(M,N)$. Then there is a polynomial
\[P \in \Hom_R\bigl(F_R(M),F_R(N)\bigr)\bigl[T_1,\dotsc,T_d\bigr]\]
which is homogeneous of degree $n$, such that for all $r_1,\dotsc,r_d \in R$ we have
\[F_R(r_1 \cdot f_1 + \dots + r_d  \cdot f_d) = P(r_1,\dotsc,r_d).\]
\end{lemma}
 
For example, in case of the homogeneous operation $M \mapsto M^{\otimes 2}$ of degree $n=2$ and $d=2$, we have
\[(r_1  \cdot f_1 + r_2  \cdot f_2)^{\otimes 2} = r_1^2  \cdot f_1^{\otimes 2} + r_1 r_2  \cdot (f_1 \otimes f_2) + r_1 r_2  \cdot (f_2 \otimes f_1) + r_2^2  \cdot f_2^{\otimes 2}.\]

\begin{proof}
Let $\Delta : M \to M^{\oplus d}$ be the diagonal and $\nabla : N^{\oplus d} \to N$ be the codiagonal.
Then we may factor the morphism $r_1  \cdot f_1 + \dots + r_d  \cdot f_d$ as follows: 
\[\begin{tikzcd}[column sep=60pt]
M \ar{r}{\Delta} & M^{\oplus d} \ar{r}{\bigoplus_{i=1}^{d} r_i \cdot \id_M} & M^{\oplus d} \ar{r}{\bigoplus_{i=1}^{d} f_i} & N^{\oplus d} \ar{r}{\nabla} & N
\end{tikzcd}\]
Thus, it suffices to write $F_R(\oplus_{i=1}^{d} r_i \cdot \id_{M})$ as a homogeneous polynomial in $r_1,\dotsc,r_d$ with coefficients in the ring $S\coloneqq\End_R\bigl(F_R(M^{\oplus d})\bigr)$. Using the base change $R[T_1,\dotsc,T_d] \to R$, $T_i \mapsto r_i$, we see that it suffices to consider the universal case: We have to prove that 
\[\alpha \coloneqq F_{R[T_1,\dotsc,T_d]}\bigl(\oplus_{i=1}^{d} T_i \cdot \id_{M[T_1,\dotsc,T_d]}\bigr): F_R(M^{\oplus d})\bigl[T_1,\dotsc,T_d\bigr] \to F_R(M^{\oplus d})\bigl[T_1,\dotsc,T_d\bigr]\]
is induced by a homogeneous polynomial of degree $n$ in $S[T_1,\dotsc,T_d]$. We consider the $\IN^d$-grading on $M^{\oplus d}$ for which, for every $1 \leq i \leq d$, the $i$-th summand $M$ is the homogeneous component of degree $e_i=(0,\dotsc,1,\dotsc,0)$. By Corollary \ref{grad}, the module $F_R(M^{\oplus d})$ inherits an $\IN^d$-grading. In fact, $\alpha|_{F_R(M^{\oplus d})}$ is the corresponding $R[T_1,\dotsc,T_d]$-coaction. In general, if we apply to an $\IN^d$-graded module, i.e.\ an $R[T_1,\dotsc,T_d]$-comodule, the base change along the morphism of bialgebras $R[T_1,\dotsc,T_d] \to R[T]$, $T_i \mapsto T$, we obtain the $\IN$-grading of total degrees. Applying this to $M^{\oplus d}$, we get the trivial $\IN$-grading concentrated in degree $1$. Since $F$ is homogeneous of degree $n$, we deduce that the $\IN^d$-grading on $F_R(M^{\oplus d})$ has only homogeneous components of  total degree $n$, the other ones being zero. This means that the homomorphism
\[\alpha|_{F_R(M^{\oplus d})} : F_R(M^{\oplus d})\to F_R(M^{\oplus d})[T_1,\dotsc,T_d]\]
lands inside the $R$-submodule of polynomials which are homogeneous of degree $n$ over $F_R(M^{\oplus d})$. Since there are only finitely many  $(i_1,\dotsc,i_d) \in \IN^d$ with $i_1+\cdots+i_d=n$, we conclude that $\alpha|_{F_R(M^{\oplus d})}$ is given by a polynomial over $S$ which is homogeneous of degree $n$.
\end{proof}

\begin{cor} \label{homo1}
An operation $F : \Mod \to \Mod$ is homogeneous of degree $1$ if and only if it is linear. In particular, the two definitions of $[\Mod,\Mod]_1$ in Definitions \ref{linear-def} and \ref{homo-def} coincide.
\end{cor}

\begin{proof} (Cf.\ \cite[Remark (2.3)]{M80})
The direction $\Longleftarrow$ is trivial. In order to prove $\Longrightarrow$, we apply Lemma \ref{poly}. If $f_1,f_2 : M \to N$ are two $R$-linear maps, there are two $R$-linear maps $g_1,g_2 : F_R(M) \to F_R(N)$ such that $F_R(r_1  \cdot f_1 + r_2  \cdot f_2) = r_1  \cdot g_1 + r_2  \cdot g_2$ holds for all $r_1,r_2 \in R$. For $r_1=1$, $r_2=0$ this shows $g_1 = F_R(f_1)$. Likewise, we have $g_2 = F_R(f_2)$. Thus, $F_R(r_1  \cdot f_1 + r_2  \cdot f_2) = r_1  \cdot F_R(f_1) + r_2  \cdot F_R(f_2)$ holds for all $r_1,r_2 \in R$. This means that $F_R$ is $R$-linear.
\end{proof}

Homogeneous operations of degree $0$ are easy to classify.

\begin{lemma} 
There is an equivalence of categories $[\Mod,\Mod]_0 \simeq \Mod(k)$ which maps $F$ to $F_k(k)$.
\end{lemma}

\begin{proof} (Cf.\ \cite[Remark (2.3)]{M80})
Any $k$-module $V$ induces the ``constant'' operation
\[M \mapsto V \otimes_k R\]
which is homogeneous of degree $0$. It maps $k \mapsto V$. Conversely, if $F : \Mod \to \Mod$ is an operation which is homogeneous of degree $0$, then for all $R$-modules $M$ we have $F_R(0 \cdot \id_M)=0^0 \cdot \id_{F_R(M)} = \id_{F_R(M)}$. It follows from this that $F_R(0 : N \to M)$ is inverse to $F_R(0 : M \to N)$. In particular, we have
\[F_R(M) \cong F_R(0)  \cong F_R(0 \otimes_k R) \cong F_k(k) \otimes_k R.\]
This isomorphism is natural in $M$. Besides, the coherence conditions in Definition \ref{defop} applied to $k \to R \to S$ show that this defines an isomorphism of operations.
\end{proof}


\section{Linearization of operations} \label{sec:red}

In this section, we will closely follow \cite[Sections 3 and 4]{M80}. The method may be seen as a categorification of the well-known polarization or linearization procedure for homogeneous polynomials \cite[Section 3.2]{P06}.
 
\begin{rem} \label{homo2}
Let $n \in \IN$. Let $F' : \Mod^n \to \Mod$ be an operation with $n$ arguments. If $G$ is any commutative monoid, then $F'$ lifts to an operation $\gr_G\!\Mod^n \to \gr_G\! \Mod$. This is done exactly as in Corollary \ref{grad}. Namely, if $(M_1,\dotsc,M_n) \in \Mod(R)^n$, then $G$-gradings on the $M_i$ correspond to morphisms $M_i[G] \to M_i[G]$ satisfying certain properties, i.e.\ to a morphism $(M_1,\dotsc,M_n)[G] \to (M_1,\dotsc,M_n)[G]$, which induces a morphism $F'_R(M_1,\dotsc,M_n)[G] \to F'_R(M_1,\dotsc,M_n)[G]$, which in turn corresponds to some $G$-grading on $F'_R(M_1,\dotsc,M_n)$.
  
In particular, if $G=\IN^n$ and we endow each $M_i$ with the trivial $\IN^n$-grading concentrated in degree $e_i = (0,\dotsc,1,\dotsc,0)$, we obtain an $\IN^n$-grading on $F'_R(M_1,\dotsc,M_n)$ whose homogeneous components define operations
\[F'_{i_1,\dotsc,i_n} : \Mod^n \to \Mod\]
which are homogeneous of degree $(i_1,\dotsc,i_n)$ and satisfy
\[\bigoplus_{(i_1,\dotsc,i_n) \in \IN^n} F'_{i_1,\dotsc,i_n} \cong F'.\]
As in Lemma \ref{homo-krit}, homogeneity means that for all commutative $k$-algebras $R$, all $R$-modules $M_1,\dotsc,M_n$ and all elements $r_1,\dotsc,r_n \in R$ we have
\[(F'_{i_1,\dotsc,i_n})_R\bigl(r_1 \cdot \id_{M_1},\dotsc,r_n \cdot \id_{M_n}\bigr) = r_1^{i_1} \cdots r_n^{i_n} \cdot \id_{(F'_{i_1,\dotsc,i_n})_R(M_1,\dotsc,M_n)}.\]
Using the base change $R[T_1,\dotsc,T_n] \to R[T]$, $T_i \mapsto T$, we see that the associated $\IN$-grading on $F'_R(M_1,\dotsc,M_n)$ given by total degrees is induced by the trivial $\IN$-grading on $(M_1,\dotsc,M_n)$ where each $M_i$ has degree $1$. Analogous remarks hold for $\Mod_{\fg}$.
\end{rem}

\begin{defi} \label{linearization}
Let $n \in \IN$. Let $F : \Mod \to \Mod$ be an operation which is homogeneous of degree $n$. We define the operation $F' : \Mod^n \to \Mod$ by
\[F'_R(M_1,\dotsc,M_n) \coloneqq F_R(M_1 \oplus \dots \oplus M_n),\]
equipped with the evident base change isomorphisms. We now apply Remark \ref{homo2} and construct a decomposition $F' \cong \bigoplus_{i_1,\dotsc,i_n} F'_{i_1,\dotsc,i_n}$ into homogeneous operations. Here, the only non-trivial operations are those with $i_1+\dots+i_n=n$; this is because the associated $\IN$-grading on $F'_R(M_1,\dotsc,M_n)=F_R(M_1 \oplus \dots \oplus M_n)$ is induced by the trivial $\IN$-grading on $M_1 \oplus \dots \oplus M_n$ concentrated in degree $1$ and $F$ is assumed to be homogeneous of degree $n$. In particular, we define the operation
\[L_F \coloneqq F'_{1,\dotsc,1} : \Mod^n \to \Mod\]
and call $L_F$ the \emph{linearization} of $F$. Explicitly, an element $u \in F_R(M_1 \oplus \dots \oplus M_n)$ belongs to $(L_F)_R(M_1,\dotsc,M_n)$ if and only if the endomorphism
\[F_R(M_1 \oplus \dots \oplus M_n)\bigl[T_1,\dotsc,T_n\bigr] \to F_R(M_1 \oplus \dots \oplus M_n)\bigl[T_1,\dotsc,T_n\bigr]\]
which is induced by the endomorphism $\bigoplus_{i=1}^{n} T_i \cdot \id_{M_i[T_1,\dotsc,T_n]}$ maps $u \cdot 1$ to $u \cdot T_1 \cdots T_n$. We make similar definitions for operations $F : \Mod_{\fg} \to \Mod_{\fg}$.
\end{defi}

\begin{ex}
For the operation $F = \Lambda^2$, which is homogeneous of degree $2$, the homogeneous components of $F'(M,N)=\Lambda^2(M \oplus N)$ are $F'(M,N)_{2,0} = \Lambda^2(M)$, $F'(M,N)_{1,1} = L_F(M,N) = M \otimes N$ and $F'_{0,2}(M,N) = \Lambda^2(N)$. More generally, the linearization of $\Lambda^n$ is the $n$-fold tensor product.
\end{ex}

\begin{defi} \label{action}
Let $F : \Mod_{\fg} \to \Mod_{\fg}$ be an operation which is homogeneous of degree $n$. Let $F'$ and $L_F$ be defined as in  Definition \ref{linearization}. Let $\sigma \in \Sigma_n$ be a permutation and let $M_1,\dotsc,M_n$ be a sequence of $R$-modules. There is an isomorphism
\[\tilde{\sigma} : M_1 \oplus \dots \oplus M_n \to M_{\sigma(1)} \oplus \dots \oplus M_{\sigma(n)}\]
characterized by $\tilde{\sigma} \circ \iota_{\sigma(i)} = \iota_i$, and therefore an isomorphism (denoted by the same symbol)
\[\tilde{\sigma} : F'_R(M_1,\dotsc,M_n) \to F'_R(M_{\sigma(1)},\dotsc,M_{\sigma(n)}).\]
It is easily checked that this restricts to an isomorphism
\[\tilde{\sigma} : (L_F)_R(M_1,\dotsc,M_n) \to (L_F)_R(M_{\sigma(1)},\dotsc,M_{\sigma(n)}).\]
Basically, this is because $T_1 \cdot \dots \cdot T_n$ is a symmetric polynomial. In particular, for every $R$-module $M$, the symmetric group $\Sigma_n$ acts \emph{from the right} on the $R$-module $(L_F)_R(M,\dotsc,M)$. In particular, we obtain a right $k[\Sigma_n]$-module structure on the $k$-module $(L_F)_k(k,\dotsc,k)$. This right $k[\Sigma_n]$-module will be denoted by $\V_F$. Every morphism of operations $F \to G$ restricts to a morphism of operations $L_F \to L_G$ and then to a morphism of right $k[\Sigma_n]$-modules $\V_F \to \V_G$. This defines a functor
\[[\Mod_{\fg},\Mod_{\fg}]_n \to \Mod_{\fg}\bigl(k[\Sigma_n]\bigr),\quad F \mapsto \V_F.\]
\end{defi}

\begin{defi} \label{action2}
Conversely, let $V$ be a finitely generated right $k[\Sigma_n]$-module. We define the corresponding \emph{Schur operation} by
\[\SS_V : \Mod_{\fg} \to \Mod_{\fg},\, M \mapsto V\otimes_{k[\Sigma_n]} M^{\otimes n},\]
equipped with the evident base change isomorphisms. Here, $\Sigma_n$ acts \emph{from the left} on $M^{\otimes n}$ by
\[\sigma \cdot (m_1 \otimes \dots \otimes m_n) \coloneqq m_{\sigma^{-1}(1)} \otimes \dots \otimes m_{\sigma^{-1}(n)}.\]
Observe that $\SS_V$ is homogeneous of degree $n$, and that every homomorphism of right $k[\Sigma_n]$-modules $V \to W$ induces a morphism $\SS_V \to \SS_W$ of operations. This defines a functor
\[\Mod_{\fg}\bigl(k[\Sigma_n]\bigr) \to [\Mod_{\fg},\Mod_{\fg}]_n,\quad V \mapsto \SS_V.\]
\end{defi}

\begin{thm} \label{Vclass}
Let $k$ be a commutative $\IQ$-algebra. Then for every finitely generated right $k[\Sigma_n]$-module $V$ there is an isomorphism of right $k[\Sigma_n]$-modules
\[V \iso \V_{\SS_V}.\]
Moreover, it is natural in $V$.
\end{thm}

\begin{proof}
Let $v \in V$ and consider the element
\[\alpha(v)  \coloneqq  v \otimes (e_1 \otimes \dots \otimes e_n) \in V \otimes_{k[\Sigma_n]} (k \oplus \dots \oplus k)^{\otimes n} = (\SS_V')_k(k,\dotsc,k).\]
We claim that it is contained in the linearization $(L_{\SS_V})_k(k,\dotsc,k)$. In fact, the image under the $k[T_1,\dotsc,T_n]$-coaction is equal to
\[v \otimes (e_1 \cdot T_1 \otimes \dots \otimes e_n \cdot T_n) = \alpha(v) \cdot T_1 \cdots T_n.\]
Clearly, $\alpha(v)$ depends linearly on $v$. If $\sigma \in \Sigma_n$ is a permutation, then
\begin{align*}
\alpha(v \sigma) & =v\sigma \otimes (e_1 \otimes \dots \otimes e_n) =  v \otimes \sigma(e_1 \otimes \dots \otimes e_n) = v \otimes (e_{\sigma^{-1}(1)} \otimes \dots \otimes e_{\sigma^{-1}(n)}) \\ &  = v \otimes (\tilde{\sigma}(e_1)  \otimes \dots \otimes \tilde{\sigma}(e_n)) =  \alpha(v) \tilde{\sigma}.
\end{align*}
Thus, we obtain a homomorphism of right $k[\Sigma_n]$-modules
\[\alpha : V \to \V_{\SS_V},\]
which is clearly natural in $V$. We will show that $\alpha$ is an isomorphism in the case $V=k[\Sigma_n]$ first. In this case, $\SS_V$ identifies with the operation $M \mapsto M^{\otimes n}$, and hence $\SS'_V$ identifies with the operation
\[(M_1,\dotsc,M_n) \mapsto \bigoplus_{1 \leq i_1,\dotsc,i_n \leq n} M_{i_1} \otimes \dots \otimes M_{i_n}.\]
The $R[T_1,\dotsc,T_n]$-coaction maps $u \in M_{i_1} \otimes \dots \otimes M_{i_n}$ to $u \cdot T_{i_1} \cdots T_{i_n}$. Thus, $L_{\SS_V}$ identifies with the operation
\[(M_1,\dotsc,M_n) \mapsto \bigoplus_{\sigma \in \Sigma_n} M_{\sigma(1)} \otimes \dots \otimes M_{\sigma(n)}.\]
In particular, there is an isomorphism $k[\Sigma_n] \iso  (L_{\SS_V})_k(k,\dotsc,k)$. One checks that it coincides with $\alpha$. A similar calculation works for the case  $V=k[\Sigma_n] \otimes_k N$ for some finitely generated $k$-module $N$.

To treat the general case, we use the fact that $\IQ$ is a splitting field for $\Sigma_n$ \cite[Corollary 4.16]{L17}, which implies that there is an isomorphism of $\IQ$-algebras $\IQ[\Sigma_n] \iso \prod_{i=1}^{r} M_{n_i}(\IQ)$ for some sequence of natural numbers $n_1,\dotsc,n_r$. Here, $M_{n_i}(\IQ)$, as a submodule of $\IQ[\Sigma_n]$, is the isotypical component of an irreducible $\IQ[\Sigma_n]$-module $V_i$. The isomorphism induces an isomorphism $k[\Sigma_n] \iso \prod_{i=1}^{r} M_{n_i}(k)$. From this it follows that every (finitely generated) $k[\Sigma_n]$-module is isomorphic to $\bigoplus_{i=1}^{r} V_i \otimes_k N_i$ for some sequence of (finitely generated) $k$-modules $N_1,\dotsc,N_r$. Each $V_i$ is a direct summand of $\IQ[\Sigma_n]$, so that each $V_i \otimes_k N_i$ is a direct summand of $k[\Sigma_n] \otimes_k N_i$. Since we already know that $\alpha$ is an isomorphism for $k[\Sigma_n] \otimes_k N_i$, and both functors $V \mapsto \SS_V$ and $F \mapsto \V_F$ are additive, the general case follows.
\end{proof}

\begin{thm} \label{Fclass}
Let $k$ be a commutative $\IQ$-algebra. Then for every homogeneous operation $F : \Mod_{\fg} \to \Mod_{\fg}$ over $k$ of degree $n$ there is an isomorphism of operations
\[\SS_{\V_F} \iso F.\]
Moreover, it is natural in $F$.
\end{thm}

\begin{proof}
Let $R$ be a commutative $k$-algebra. In Lemma \ref{poly} we have proven that the functor $F_R : \Mod_{\fg}(R) \to \Mod_{\fg}(R)$ is a polynomial functor which is homogeneous of degree $n$ in the sense of \cite[Sections 1 and 2]{M80}. It follows from \cite[Theorem 4.10]{M80} that there is a natural isomorphism
\[(L_F)_R(M,\dotsc,M)^{\Sigma_n} \iso F_R(M).\]
Specifically, it is given by the composition
\[(L_F)_R(M,\dotsc,M)^{\Sigma_n} \hookrightarrow (L_F)_R(M,\dotsc,M) \hookrightarrow F_R(M \oplus \dots \oplus M) \xrightarrow{F_R(\nabla)} F_R(M),\]
where $\nabla : M^{\oplus n} \to M$ is the codiagonal. The inverse has a similar description. Since $L_F$ is homogeneous of degree $(1,\dotsc,1)$, we may apply Corollary \ref{homo1} to each variable to deduce that $L_F$ is multilinear in the sense of Definition \ref{multilindef}. Thus, Theorem \ref{multilin} shows that for all $R$-modules $M_1,\dotsc,M_n$ the canonical homomorphism
\[(L_F)_k(k,\dotsc,k) \otimes_k (M_1 \otimes_R \dots \otimes_R M_n) \to (L_F)_R(M_1,\dotsc,M_n)\]
is an isomorphism. In particular, there is an isomorphism
\[(L_F)_k(k,\dotsc,k) \otimes_k M^{\otimes n} \iso (L_F)_R(M,\dotsc,M).\]
This is an isomorphism of right $k[\Sigma_n]$-modules when we use the right action of $\Sigma_n$ on $(L_F)_R(M,\dotsc,M)$ (resp.\ $(L_F)_k(k,\dotsc,k)$) from Definition \ref{action} and the right action of $\Sigma_n$ on $M^{\otimes n}$ which is induced by the left action from Definition \ref{action2} via pullback with $\sigma \mapsto \sigma^{-1}$. Since $n!$, the order of $\Sigma_n$, is invertible in $k$, it follows that
\begin{align*}
(L_F)_R(M,\dotsc,M)^{\Sigma_n} &\cong (L_F)_R(M,\dotsc,M)_{\Sigma_n} \cong \bigl((L_F)_k(k,\dotsc,k) \otimes_k M^{\otimes n}\bigr)_{\Sigma_n}
\\& \cong (L_F)_k(k,\dotsc,k) \otimes_{k[\Sigma_n]} M^{\otimes n},
\end{align*}
where in the last tensor product we use the left action of $\Sigma_n$ on $M^{\otimes n}$ again. Composing all these isomorphisms yields a natural isomorphism
\[\V_F \otimes_{k[\Sigma_n]} M^{\otimes n} \iso F_R(M).\]
One checks that this is, in fact, a morphism of operations.
\end{proof}

\begin{thm} \label{class}
Let $n \in \IN$. Let $k$ be a commutative $\IQ$-algebra. Then $V \mapsto \SS_V$ defines an equivalence of categories
\[\Mod_{\fg}\bigl(k[\Sigma_n]\bigr) \simeq [\Mod_{\fg},\Mod_{\fg}]_n.\]
\end{thm}

\begin{proof}
This follows from Theorems \ref{Fclass} and \ref{Vclass}.
\end{proof}

\begin{rem}
One can use the classification of representations of $\Sigma_n$ in order to make the classification of homogeneous operations even more explicit \cite[Chapter 8]{F97}. For example, the three irreducible Schur operations which are homogeneous of degree $3$ are $\Sym^3$, $\Lambda^3$ and
\[M \mapsto (\Lambda^2(M) \otimes M) / \bigl\langle (a \wedge b) \otimes c +  (b \wedge c) \otimes a + (c \wedge a) \otimes b : a,b,c \in M \bigr\rangle .\]
Every other homogeneous operation of degree $3$ on $\Mod_{\fg}$ is a linear combination of such operations.
\end{rem}
 
\begin{thm} \label{main}
Let $k$ be a commutative $\IQ$-algebra. Then $(V_n)_{n \in \IN} \mapsto \bigoplus_{n \in \IN} \SS_{V_n}$ induces an equivalence of categories
\[\bigoplus_{n \in \IN} \Mod_{\fg}\bigl(k[\Sigma_n]\bigr) \simeq [\Mod_{\fg},\Mod_{\fg}]_{\bound}.\]
\end{thm}

\begin{proof}
This follows from Theorem \ref{class} and Remark \ref{bounded2}.
\end{proof}

\begin{rem} \label{mainn}
Theorem \ref{main} remains true for $\Mod_{\fp}$ using the same proof. Besides, a very similar proof can be used to show, for every $m \in \IN$,
\[\bigoplus_{n \in \IN^m} \Mod_{\fg}\bigl(k[\Sigma_{n_1} \times \dots \times \Sigma_{n_m}]\bigr) \simeq [\Mod_{\fg}^m,\Mod_{\fg}]_{\bound}.\]
The equivalence maps $(V_n)_{n \in \IN^m}$ to the Schur operation with $m$ arguments
\[(M_1,\dotsc,M_m) \mapsto \bigoplus_{n \in \IN^m} V_n \otimes_{k[\Sigma_{n_1} \times \dots \times \Sigma_{n_m}]} M_1^{\otimes n_1} \otimes \dots \otimes M_m^{\otimes n_m}.\]
\end{rem}

We can also describe the full category of operations over a field $k$ of characteristic zero. If $V$ is a finitely generated right $k[\Sigma_n]$-module, then  $V$ is isomorphic to a finite direct sum of copies of Specht modules $V_{\lambda}$ associated to partitions $\lambda$ of $n$ \cite[Lecture 4]{FH91}. We will say that \emph{$\lambda$ appears in $V$} if the multiplicity $m_{\lambda}(V)$ of $V_{\lambda}$ in $V$ is positive.
 
\begin{thm} \label{mainfull}
Let $k$ be a field of characteristic zero. Then $F \mapsto (\V_{F_n})_{n \in \IN}$ induces an equivalence of categories between $[\Mod_{\fg},\Mod_{\fg}]$ and the category of sequences of finitely generated right $k[\Sigma_n]$-modules $V_n$ such that, for every $d \in \IN$, there are only finitely many partitions of length $\leq d$ that appear in one of the $V_n$.
\end{thm}
 
\begin{proof}
By Corollary \ref{decomp2} and Theorem \ref{Fclass}, the category $[\Mod_{\fg},\Mod_{\fg}]$ is equivalent to the category of sequences of finitely generated right $k[\Sigma_n]$-modules $V_n$ such that for every finitely generated $R$-module $M$ almost all of the $R$-modules $V_n \otimes_{k[\Sigma_n]} M^{\otimes n}$ vanish. Since Schur functors preserve epimorphisms, and free $R$-modules are base changes of free $k$-modules, it suffices to consider the case $R=k$ and $M=k^{\oplus d}$ for some $d \in \IN$. We have
\[V_n \otimes_{k[\Sigma_n]} M^{\otimes n} \cong \bigoplus_{\lambda \text{ partition of } n} \bigl(V_{\lambda} \otimes_{k[\Sigma_n]} M^{\otimes n}\bigr)^{\oplus m_{\lambda}(V_n)}.\]
This vanishes if and only if for all $\lambda$ appearing in $V_n$ we have $V_{\lambda} \otimes_{k[\Sigma_n]} M^{\otimes n}=0$. By \cite[Theorem 6.3 (1)]{FH91}, this happens if and only if the length of $\lambda$ is $>d$. Thus, for fixed $d \in \IN$, the condition is that almost all $n$ have the property that all partitions appearing in $V_n$ are of length $>d$. This is logically equivalent to the condition that there are only finitely many partitions of length $\leq d$ that appear in one of the $V_n$.
\end{proof}
 
\begin{ex}
The partitions $(~)$, $(1)$, $(2)$, $\dotsc$ of length $1$ do not induce an operation on $\Mod_{\fg}$. In fact, they correspond to the symmetric algebra on $\Mod$. In contrast, the partitions $(~)$, $(1)$, $(1,1)$, $(1,1,1)$, $\dotsc$ have exactly one partition of length $d$, for each $d \in \IN$, and therefore do induce an operation on $\Mod_{\fg}$, namely the exterior algebra. We could also allow that, for instance, the $n$-th partition has multiplicity $n$. Another positive example is the sequence of partitions (sorted by length, not degree) $(~)$, $(1)$, $(2)$, $(1,1)$, $(2,1)$, $(2,2)$, $(1,1,1)$, $(2,1,1)$, $(2,2,1)$, $(2,2,2)$, $\dotsc$.
\end{ex}


\section{Appendix on constructive algebra} \label{sec:con}

In our proof of Theorem \ref{lin} we implicitly used both the axiom of choice and the law of the excluded middle. In this section we will give a proof of Theorem \ref{lin} and hence of Theorem \ref{multilin} which works in constructive algebra. This means that we will not use the law of the excluded middle. In particular, by Diaconescu's Theorem, the axiom of choice will not be available. 

Our first result is a constructive version of Grothendieck's Generic Freeness Lemma \cite[Theorem 14.4]{E95}. Actually, it is only the special case for $R$-modules; the general statement also involves $R$-algebras. We will include the proof because we could not find a proper reference for precisely this version.
 
\begin{lemma} \label{generic-free}
Let $R$ be a commutative reduced ring. Let $M$ be a finitely generated $R$-module. Assume that $f \in R$ has the following property: every $g \in \langle f \rangle$ such that $M[g^{-1}]$ is free over $R[g^{-1}]$ satisfies $g=0$. Then we have $f=0$. In particular, for $f \coloneqq 1$, if every $g \in R$ such that $M[g^{-1}]$ is free over $R[g^{-1}]$ satisfies $g=0$, then $R=0$.
\end{lemma}
 
\begin{proof}
By considering the $R[f^{-1}]$-module $M[f^{-1}]$ and using $R[f^{-1}]=0 \Leftrightarrow f=0$, it suffices to consider the case $f=1$. Thus, every $g \in R$ such that $M[g^{-1}]$ is free over $R[g^{-1}]$ satisfies $g=0$, and our goal is to prove $R=0$. Let $m_1,\dotsc,m_n$ be a generating system of $M$. We will argue by induction on $n \in \IN$. The case $n=0$ is trivial. Let $n \geq 1$. We will prove that $m_1,\dotsc,m_n$ is a basis. So let us assume $g_1 m_1 + \dots + g_n m_n = 0$ with $g_1,\dotsc,g_n \in R$. Then for every $1 \leq i \leq n$, the localization $M[g_i^{-1}]$ is generated by $m_1,\dotsc,\widehat{m_i},\dotsc,m_n$, and it satisfies the same assumptions as $M$. Therefore, by induction hypothesis, we conclude $R[g_i^{-1}]=0$ and hence $g_i=0$. This proves that $M$ is free over~$R$. Hence, $1=0$ by assumption.
\end{proof}

\begin{rem}
In the presence of the law of the excluded middle, i.e.\ in classical mathe\-matics, the special case $f=1$ in Lemma \ref{generic-free} says that for $R \neq 0$ there is some element $g \in R \setminus \{0\}$ such that $M[g^{-1}]$ is free over $R[g^{-1}]$. This is the more common formulation of generic freeness. However, this statement is not valid in constructive mathematics. Geometrically, Lemma \ref{generic-free} says that there is a dense open subset $U \subseteq \Spec(R)$ such that $M^{\sim}|_U$ is locally free. In the internal language of the topos of sheaves on $\Spec(R)$ \cite{M74}, this simply says that $M^{\sim}$ is \emph{not not} free. This property may be deduced from the fact that the structure sheaf $R^{\sim}$ is a field from the internal perspective and the observation that finitely generated vector spaces are \emph{not not} free by the usual argument in linear algebra \cite[Lemma 5.9]{B16}. In fact, this argument has just been repeated in our proof of Lemma \ref{generic-free} using the external language.
\end{rem}

\begin{lemma} \label{free-locus}
Let $R$ be a commutative ring. Let $M$ and $N$ be two finitely generated $R$-modules with the following property: If $S$ is a commutative $R$-algebra such that $M \otimes_R S$ is free over $S$, then $N \otimes_R S = 0$. Then we may conclude $N=0$.
\end{lemma}

\begin{proof}
Consider the reduced commutative ring $R'\coloneqq R/\sqrt{\Ann(N)}$. The $R'$-modules $M' \coloneqq  M \otimes_R R'$ and $N' \coloneqq  N \otimes_R R'$ satisfy the same assumption as the $R$-modules $M$ and $N$. Assume that $f' \in R'$ has the property that $M'[f'^{-1}]$ is free over $R'[f'^{-1}]$. We will prove $f'=0$. By assumption, we have $N'[f'^{-1}]=0$. Because $N'$ is finitely generated, there is some $k \in \IN$ such that $f'^k N'=0$. Choose a preimage $f \in R$. Then we have $f^k N \subseteq \sqrt{\Ann(N)} N$. Since $f'=0$ is equivalent to $f'^k=0$, we might as well assume that $f N \subseteq \sqrt{\Ann(N)} N$ holds. By \cite[Proposition 2.4]{AM69}, applied to the endomorphism $f \cdot \id_N : N \to N$, there are elements $r_0,\dotsc,r_{n-1} \in \sqrt{\Ann(N)}$ such that
\[f^n + r_{n-1} f^{n-1} + \dots + r_1 f + r_0 \in \Ann(N).\]
This implies $f^n \in \sqrt{\Ann(N)}$ and therefore $f'=0$. We have proven that every $f' \in R'$ such that $M'[f'^{-1}]$ is free over $R'[f'^{-1}]$ satisfies $f'=0$. By Lemma \ref{generic-free}, we conclude $R'=0$. This means $1 \in \Ann(N)$, i.e.\ $N=0$.
\end{proof}

\begin{ex}
Lemma \ref{free-locus} remains true if just $N$ is assumed to be finitely generated, but in general it does not hold. Let $f \in R$ be a regular element, $M \coloneqq R/fR$ and 
\[N \coloneqq  \colim_{n>0}\, R/f^n R.\]
The transition maps are $[x] \mapsto [fx]$. If $M \otimes_R S = S/fS$ is free over $S$, this means that $f \in \ker(R \to S)$ or $f \in S^{\times}$ holds. Thus, $S$ is an $R/fR$-algebra or an $R[f^{-1}]$-algebra. In the first case, we have $N \otimes_R R/fR = N/fN = 0$ and hence $N \otimes_R S = 0$. In the second case, we have $N \otimes_R R[f^{-1}]=N[f^{-1}]=0$ and hence $N \otimes_R S = 0$. But $N=0$ holds if only if $f$ is a unit. So we may take $R \coloneqq \IZ$ and $f \coloneqq 2$.
\end{ex}

\begin{proof}[Constructive proof of Theorem \ref{lin}]
We only have to give a constructive proof of the statement that a linear operation $G:\Mod_{\fg} \to \Mod_{\fg}$ vanishes when it vanishes on free modules, because the rest of the proof was constructive anyway. Let $M$ be a finitely generated $R$-module. If $S$ is a commutative $R$-algebra such that $M \otimes_R S$ is free over $S$, then we have
\[G_R(M) \otimes_R S \cong G_S(M \otimes_R S)=0.\]
Thus, Lemma \ref{free-locus} applies and yields $G_R(M)=0$.
\end{proof}


\phantom{-}

\end{document}